\documentclass[12pt]{article}%
\usepackage{amssymb}
\usepackage{eurosym}
\usepackage{amsfonts}
\usepackage{amsmath}
\usepackage{graphicx}
\usepackage{epstopdf}
\usepackage[⟨options⟩]{overpic}%
\providecommand{\U}[1]{\protect\rule{.1in}{.1in}}
\newtheorem{theorem}{Theorem}
\newtheorem{acknowledgement}[theorem]{Acknowledgement}

\newtheorem{lemma}[theorem]{Lemma}

\newtheorem{proposition}[theorem]{Proposition}
\newtheorem{remark}[theorem]{Remark}

\newenvironment{proof}[1][Proof]{\noindent\textbf{#1.} }{\ \rule{0.5em}{0.5em}}
\ifx\pdfoutput\relax\let\pdfoutput=\undefined\fi
\newcount\msipdfoutput
\ifx\pdfoutput\undefined\else
\ifcase\pdfoutput\else
\ifx\paperwidth\undefined\else
\ifdim\paperheight=0pt\relax\else\pdfpageheight\paperheight\fi
\ifdim\paperwidth=0pt\relax\else\pdfpagewidth\paperwidth\fi
\fi\fi\fi
\begin{document}

\title{Algebraic Equation and Quadratic\ Differential Related to Generalized Bessel
Polynomials with Varying Parameters}
\author{Mohamed Jalel Atia and Faouzi Thabet}
\maketitle

\begin{abstract}
The limiting set of zeros of generalized Bessel polynomials $B_{n}^{(\alpha)}$
with varying parameters depending on the degree $n$ cluster in a curve on the
complex plane, which is a finite critical trajectory of a quadratic
differential in the form $\lambda^{2}\frac{\left(  z-a\right)  \left(
z-b\right)  }{z^{4}}dz^{2}.$ The motivation of this paper is the description
of the critical graphs of these quadratic differentials. In particular, we
give a necessary and sufficient condition on the existence of short trajectories.

\end{abstract}



\section{Introduction\bigskip\ }

Let consider the quadratic differential $\varpi\left(  \lambda,a,b,z\right)  $
defined on the Riemann sphere $\widehat{%
\mathbb{C}
}$ by :
\begin{equation}
\varpi\left(  \lambda,a,b,z\right)  =\lambda^{2}\frac{\left(  z-a\right)
\left(  z-b\right)  }{z^{4}}dz^{2}, \label{genral qd}%
\end{equation}
where $a,b,$ and $\lambda$ are three non vanishing complex numbers with $a\neq
b.$ \emph{Finite critical points} of $\varpi\left(  \lambda,a,b,z\right)  $
are $a$ and $b,$ as simple zeros; while the origin is an obvious
\emph{infinite critical point,} as a pole of order $4.$ By the change of
variable $u=1/z,$ $\varpi\left(  \lambda,a,b,z\right)  $ has another infinite
critical point, as a double pole that is located at $\infty:$
\begin{equation}
\ \varpi\left(  \lambda,a,b,z\right)  =\left(  \frac{\lambda^{2}}{u^{2}%
}+\mathcal{O}(u^{-1})\right)  du^{2},\quad u\rightarrow0. \label{res}%
\end{equation}
The number $\lambda^{2}$ is called the \emph{residue} of $\varpi\left(
\lambda,a,b,z\right)  $ at the double pole $\infty.$ All other points of
$\widehat{%
\mathbb{C}
}$ are \emph{regular.}

Horizontal trajectories (or just trajectories) are the zero loci of the
equation%
\begin{equation}
\mathcal{\Im}\int^{z}\lambda\frac{\sqrt{\left(  t-a\right)  \left(
t-b\right)  }}{t^{2}}\,dt=\text{\emph{const}}, \label{re}%
\end{equation}
or equivalently%
\begin{equation}
\lambda^{2}\frac{\left(  z-a\right)  \left(  z-b\right)  }{z^{4}}dz^{2}>0.
\label{>0}%
\end{equation}
Vertical (or, orthogonal) trajectories are obtained by replacing $\Im$ and $>$
by $\Re$ and $<$ respectively in the equations (\ref{re}) and (\ref{>0}).
Horizontal and vertical trajectories of $\varpi\left(  \lambda,a,b,z\right)  $
produce two pairwise orthogonal foliations of the Riemann sphere $\widehat{%
\mathbb{C}
}.$

A trajectory going through $a$ or $b$ is called \emph{critical trajectory}. If
it starts and ends at $a$ or $b$ it is called \emph{finite critical trajectory
}or\emph{ short trajectory}. If it starts at $a$ or $b$ but tends either to
the origin or to infinity, we call it \emph{infinite critical trajectory}. The
closure of the set of all critical trajectories of $\varpi\left(
\lambda,a,b,z\right)  $ is called the \emph{critical graph} of $\varpi\left(
\lambda,a,b,z\right)  .$ It will be denoted by $\Gamma\left(  \lambda
,a,b,z\right)  .$

In section 2, we give a full description of the critical graph of
$\varpi\left(  \lambda,a,b,z\right)  .$ In particular, Proposition \ref{main}
gives a necessary and sufficient condition on the existence of finite critical trajectories.

Quadratic differentials (\ref{genral qd}) appear in the investigation of the
existence of a compactly supported positive measure $\mu$ (in the literature,
such a measure $\mu$ is called \emph{positive mother-body measure}) whose
Cauchy transform $\mathcal{C}_{\mu}\left(  z\right)  $ satisfies in some
non-empty subset of $%
\mathbb{C}
$ an irreducible algebraic equation in the form :
\begin{equation}
z^{2}\mathcal{C}^{2}\left(  z\right)  +\left(  pz+q\right)  \mathcal{C}\left(
z\right)  +r=0, \label{eq alg}%
\end{equation}
where $p,$ $q$ and $r$ are three non vanishing complex numbers.

Proposition \ref{first prop} below answers this question. In section 3, we
make the connection with generalized Bessel polynomials $B_{n}^{\left(
\alpha\right)  }$ when the parameters determining these polynomials are
complex and depend linearly on the degree.

\bigskip

\section{Critical graph of $\varpi\left(  \lambda,a,b,z\right)  $}

From each zero, $a$ and $b,$ there emanate $3$ critical trajectories spacing
under equal angle $2\pi/3.$ Since the origin is a pole of order $4,$ there are
two opposite asymptotic directions (called \emph{critical directions}) :
\[
D_{k}=\left\{  z\in%
\mathbb{C}
:\arg z=\frac{\arg\left(  \lambda^{2}ab\right)  }{2}+k\pi\right\}  ;k=0,1,
\]
and a neighborhood $\mathcal{V}$ of the origin, such that each trajectory
entering $\mathcal{V}$ stays in $\mathcal{V}$ and tends to the origin
following a certain critical direction; see Fig.\ref{1}. The behavior of
trajectories near the double pole $\infty$ depends on the residue $\lambda
^{2}$ defined in (\ref{res}); they have the radial, circle, or log-spiral
forms, respectively if $\lambda\in%
\mathbb{R}
^{\ast},$ $\lambda\in i%
\mathbb{R}
^{\ast},$ or $\lambda^{2}\notin%
\mathbb{R}
;$ see Fig.\ref{2}.

Jenkins Theorem \cite[Theorem 3.5]{MR1929066} asserts that the critical graph
$\Gamma\left(  \lambda,a,b,z\right)  $ splits the Riemann sphere into a finite
number of domains called the \emph{domain configurations} of $\varpi\left(
\lambda,a,b,z\right)  .$ There are in general five possible types of domain
configurations : \emph{Half-plane, Strip,} \emph{Circle, Ring, }and\emph{
Dense domains.} Thanks to the Jenkins Three-pole Theorem \cite[Theorem
15.2]{Strebel:84}, since $\varpi\left(  \lambda,a,b,z\right)  $ has only two
poles, it cannot have the last type.
\begin{figure}[tbh]
\begin{minipage}[b]{0.4\linewidth}
		\centering\includegraphics[scale=0.3]{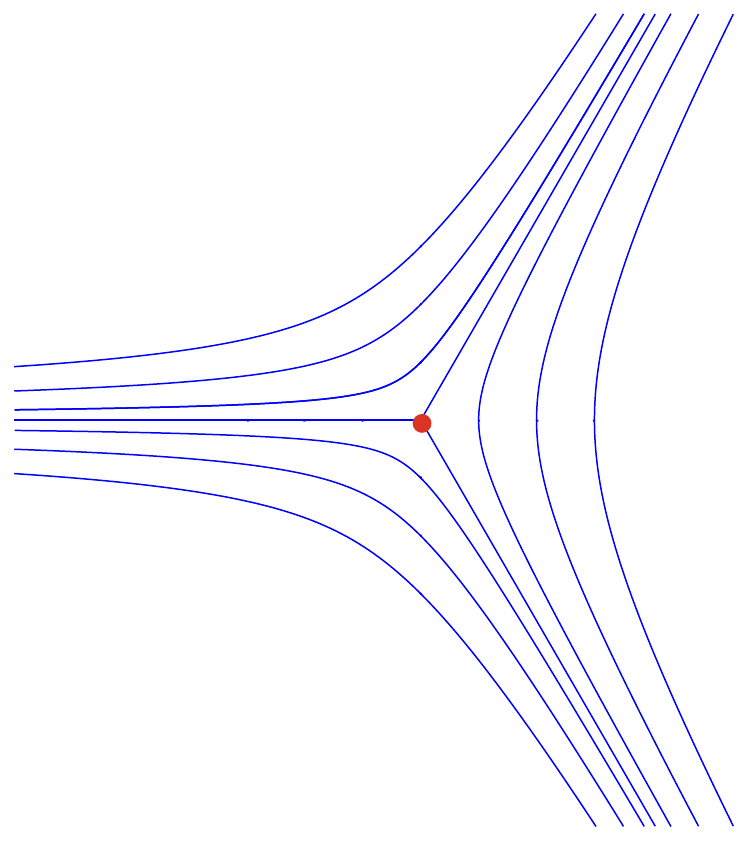}
		\hfill\end{minipage}\hfill\begin{minipage}[b]{0.47\linewidth}
		\centering\includegraphics[scale=0.38]{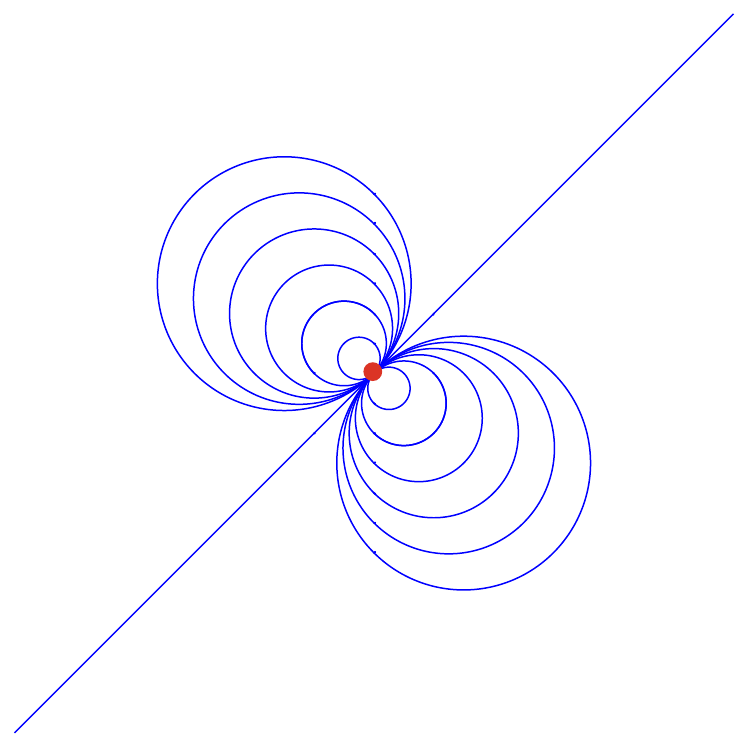}
		\hfill\end{minipage}\caption{Local behavior of the trajectories near a
simple zero (left) and a 4th order pole (right) }%
\label{1}%
\end{figure}\begin{figure}[tbh]
\begin{minipage}[b]{0.3\linewidth}
		\centering\includegraphics[scale=0.25]{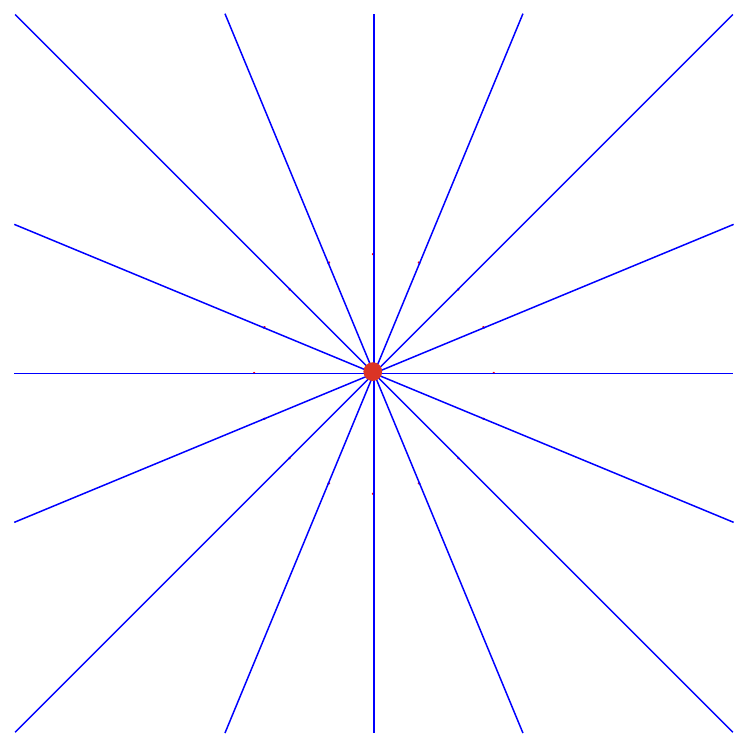}
		\hfill\end{minipage}\hfill\begin{minipage}[b]{0.3\linewidth}
		\centering\includegraphics[scale=0.25]{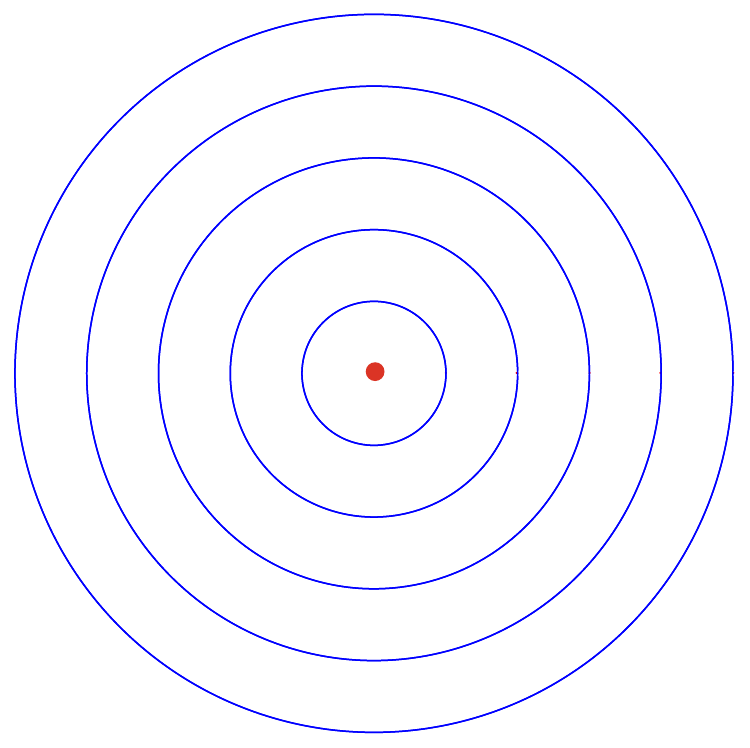}
		\hfill\end{minipage}\hfill\begin{minipage}[b]{0.3\linewidth}
		\centering\includegraphics[scale=0.25]{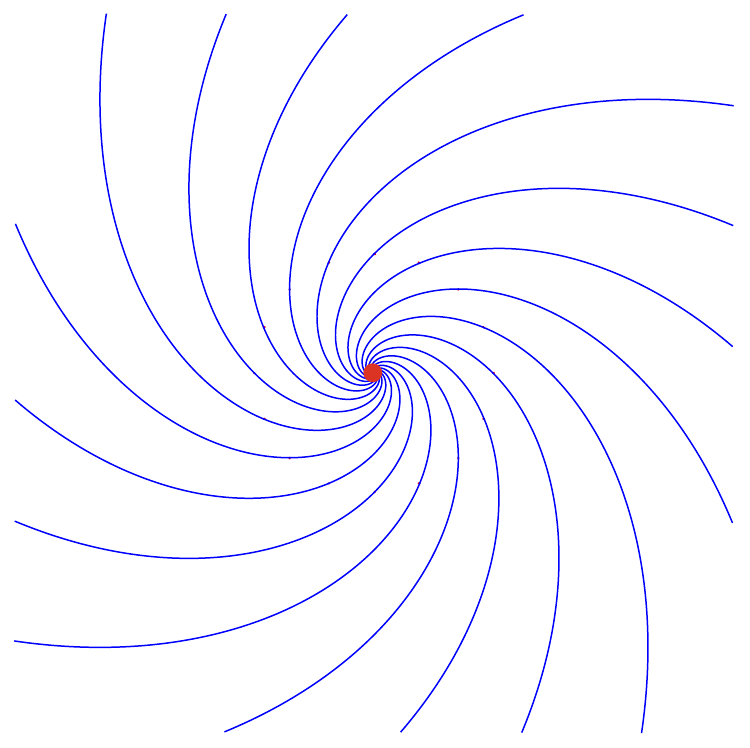}
		\hfill\end{minipage}\caption{Local behavior of the trajectories near
$\infty$ for $\lambda\in\mathbb{R}^{\ast}$ (left)$,$ $\lambda\in
i\mathbb{R}^{\ast}$ (middle), and $\lambda^{2}\notin\mathbb{R}$ (right).}%
\label{2}%
\end{figure}

The following Theorem \cite[Theorem 14.1]{Strebel:84} is a powerful tool to
clarify some facts about the global structure of the trajectories :

\begin{theorem}
\label{teich} Let $\Omega$ be a $\varpi$-polygon : a simply connected domain
in $%
\mathbb{C}
$ bounded only by segments of horizontal and/or vertical trajectories of
$\varpi\left(  \lambda,a,b,z\right)  $ (and their endpoints). Let $z_{j}$ be
the critical points of $\varpi\left(  \lambda,a,b,z\right)  $ on the boundary
$\partial\Omega$ of $\Omega,$ with multiplicities $n_{j},$ and let $\theta
_{j}$ be the corresponding interior angles with vertices at $z_{j},$
respectively. Then%
\[
\sum\left(  1-\theta_{j}\dfrac{n_{j}+2}{2\pi}\right)  =2+\sum n_{i},
\]
where $n_{i}$ are the multiplicities of all critical points of $\varpi\left(
\lambda,a,b,z\right)  $ inside $\Omega.$
\end{theorem}

For more details on the theory of quadratic differentials and its
trajectories, we refer the reader to \cite{MR1929066},\cite{Strebel:84}.

A necessary condition on the existence of a short trajectory joining $a$ and
$b$ is that
\[
\mathcal{\Im}\int_{\gamma}\lambda\frac{\sqrt{\left(  t-a\right)  \left(
t-b\right)  }}{t^{2}}\,dt=0,
\]
for some Jordan arc $\gamma$ connecting $a$ and $b$ in $%
\mathbb{C}
\setminus\left\{  0\right\}  .$

The main result of this section is the following :

\begin{proposition}
\label{main} Quadratic differential $\varpi\left(  \lambda,a,b,z\right)  $ has
a finite critical trajectory connecting $a$ and $b,$ if and only if,
\begin{equation}
\Re\left(  \lambda\frac{\left(  \sqrt{a}+\sqrt{b}\right)  ^{2}}{\sqrt{ab}%
}\right)  =0,\text{ or }\Re\left(  \lambda\frac{\left(  \sqrt{a}-\sqrt
{b}\right)  ^{2}}{\sqrt{ab}}\right)  =0. \label{cond nec}%
\end{equation}

\end{proposition}

Below, given an oriented Jordan curve $\gamma$ joining $a$ and $b$ in $%
\mathbb{C}
\setminus\left\{  0\right\}  ,$ for $t\in\gamma,$ we denote by $\left(
\sqrt{\left(  t-a\right)  \left(  t-b\right)  }\right)  _{+}$ and $\left(
\sqrt{\left(  t-a\right)  \left(  t-b\right)  }\right)  _{-}$ the limits from
the $+$-side and $-$-side respectively. (As usual, the $+$-side of an oriented
curve lies to the left, and the $-$-side lies to the right, if one traverses
$\gamma$ according to its orientation.)

\begin{lemma}
\label{real} For any oriented Jordan curve $\gamma$ joining $a$ and $b$ in $%
\mathbb{C}
\setminus\left\{  0\right\}  ,$ we have:%
\[
\int_{\gamma}\dfrac{\left(  \sqrt{\left(  z-a\right)  \left(  z-b\right)
}\right)  _{+}}{z^{2}}dz=\pm\frac{i\pi}{2}\frac{\left(  \sqrt{a}\pm\sqrt
{b}\right)  ^{2}}{\sqrt{ab}},
\]
the signs $\pm$ depend on the homotopy class of $\gamma$ in $%
\mathbb{C}
\setminus\left\{  0\right\}  ,$ and the branch of the square root
$\sqrt{\left(  z-a\right)  \left(  z-b\right)  }$ defined in $%
\mathbb{C}
\setminus\gamma$ is taken with normalization condition :
\[
\sqrt{\left(  z-a\right)  \left(  z-b\right)  }\sim z,z\rightarrow\infty.
\]

\end{lemma}

\begin{proof}
With the above choice of the square root, consider%
\[
I=\int_{\gamma}\frac{\left(  \sqrt{\left(  z-a\right)  \left(  z-b\right)
}\right)  _{+}}{z^{2}}dz.
\]
Since
\[
\left(  \sqrt{\left(  t-a\right)  \left(  t-b\right)  }\right)  _{+}=-\left(
\sqrt{\left(  t-a\right)  \left(  t-b\right)  }\right)  _{-},t\in\gamma,
\]
we have
\begin{align*}
2I  &  =\int_{\gamma}\left[  \frac{\left(  \sqrt{\left(  t-a\right)  \left(
t-b\right)  }\right)  _{+}}{t^{2}}-\dfrac{\left(  \sqrt{\left(  t-a\right)
\left(  t-b\right)  }\right)  _{-}}{t^{2}}\right]  dt\\
&  =\oint_{\Gamma}\dfrac{\sqrt{\left(  z-a\right)  \left(  z-b\right)  }%
}{z^{2}}dz,
\end{align*}
where $\Gamma$ is a closed contour encircling the curve $\gamma$ once in the
clockwise direction and not encircling $z=0.$ After a contour deformation we
pick up residues at $z=0$ and at $z=\infty.$ From the expressions%
\begin{align*}
\frac{\sqrt{\left(  z-a\right)  \left(  z-b\right)  }}{z^{2}}  &
=\allowbreak\frac{\sqrt{ab}}{z^{2}}-\frac{\sqrt{ab}\left(  a+b\right)  }%
{2ab}\frac{1}{z}+\mathcal{O}\left(  1\right)  ,z\rightarrow0;\\
\frac{\sqrt{\left(  z-a\right)  \left(  z-b\right)  }}{z^{2}}  &  =\frac{1}%
{z}+\mathcal{O}\left(  z^{-2}\right)  ,z\rightarrow\infty,
\end{align*}
we get, for any choice of the square roots
\begin{align*}
I  &  =\frac{1}{2}\oint_{\gamma}\frac{\sqrt{\left(  z-a\right)  \left(
z-b\right)  }}{z^{2}}dz=\pm i\pi\left(  \underset{0}{res}+\underset{\infty
}{res}\right)  \left(  \frac{\sqrt{\left(  z-a\right)  \left(  z-b\right)  }%
}{z^{2}}\right) \\
&  =\pm i\pi\left(  \frac{\sqrt{ab}\left(  a+b\right)  -2ab}{2ab}\right)  =\pm
i\frac{\pi}{2}\frac{\left(  \sqrt{a}\pm\sqrt{b}\right)  ^{2}}{\sqrt{ab}}.
\end{align*}
We just proved the necessary condition of Proposition \ref{main}.
\end{proof}

\begin{lemma}
\label{2TRAJ TO INF}If $\Re\left(  \lambda\right)  \neq0,$ then there exists
always a critical trajectory of $\varpi\left(  \lambda,a,b,z\right)  $
diverging to infinity; if there are two, then they emanate from different zeros.
\end{lemma}

\begin{proof}
\bigskip Let us keep the notations of Lemma \ref{teich}. Suppose that
$\gamma_{1}$ and $\gamma_{2}$ are two trajectories emanating from $a,$ spacing
with angle $\theta_{a}\in\left\{  \frac{2\pi}{3},\frac{4\pi}{3}\right\}  ,$
and diverging simultaneously to the same pole $c\in$ $\left\{  \infty
,0\right\}  .$ They form an $\varpi$-polygon $\Omega$ with vertices $a$ and
$c.$

\begin{itemize}
\item If $c=\infty,$ then there is always an orthogonal trajectory
intersecting $\gamma_{1}$ and $\gamma_{2}$ at right angles; to each of these
two corners, formed in this way, it corresponds the value of $\beta_{j}%
=\frac{1}{2},$ so their sum equals $1.$ Making the intersection points
approach $c=\infty,$ we see that in the limit we can consider $\beta_{c}=1$
for $z_{c}=\infty.$ Then
\[
\sum n_{i}=-1+\left(  1-\dfrac{3\theta_{a}}{2\pi}\right)  \in\left\{
-2,-1\right\}  ;
\]
which violates Lemma \ref{teich}.

\item If $c=0,$ let $\theta_{0}\in\left\{  0,\pi,2\pi\right\}  $ be the
interior angle of the $\varpi$-polygon $\Omega$ between $\gamma_{1}$ and
$\gamma_{2}$ at $c=0$.

\begin{itemize}
\item If $\theta_{0}=0,$ then
\[
\sum n_{i}=-\dfrac{3\theta}{2\pi}\in\left\{  -2,-1\right\}  ;
\]
which cannot hold.

\item If $\theta_{0}=\pi,$ then%
\[
\sum n_{i}=1-\dfrac{3\theta}{2\pi}\in\left\{  0,-1\right\}  .
\]
Thus, $\theta=\frac{2\pi}{3}$ and $\Omega$ does not contain inside itself any
critical point; see Figure \ref{FIG2}.

\item If $\theta_{0}=2\pi,$ then%
\[
\sum n_{i}=2-\dfrac{3\theta}{2\pi}\in\left\{  0,1\right\}  .
\]
Thus, if $\theta=\frac{2\pi}{3},$ then $\Omega$ contains inside itself the
other zero $b$ and the 3 trajectories emanating from it; if $\theta=\frac
{4\pi}{3},$ then $\Omega$ does not contain inside itself any critical point;
see Figure \ref{FIG3}.
\end{itemize}
\end{itemize}
\end{proof}

\begin{figure}[h]
\begin{minipage}[b]{0.48\linewidth}
\centering \includegraphics[scale=0.4]{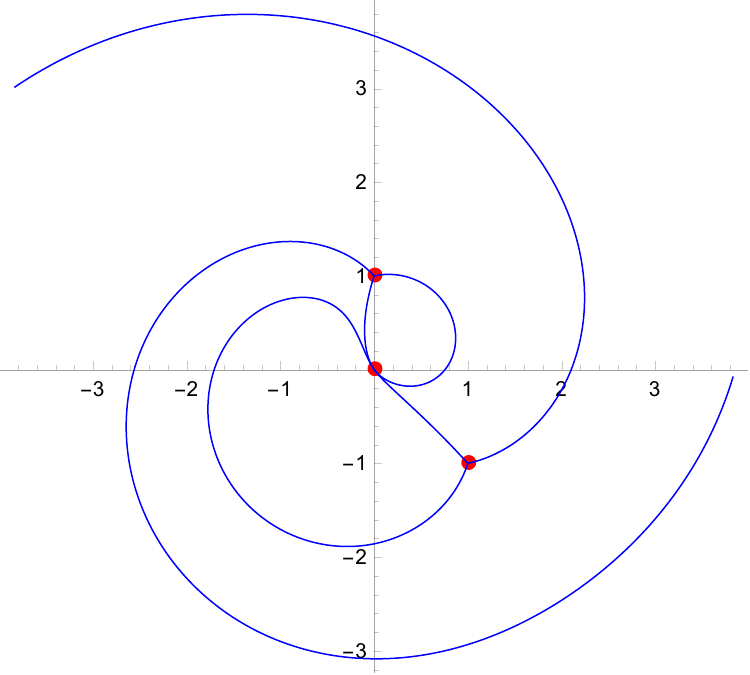}
\caption{$c=0$ and $\theta_{0}=\pi.$ Here $\lambda^{2}%
=-2-i,a=1-i,b=i$.}%
\label{FIG2}%
\end{minipage}\hfill\begin{minipage}[b]{0.48\linewidth}
\centering \includegraphics[scale=0.5]{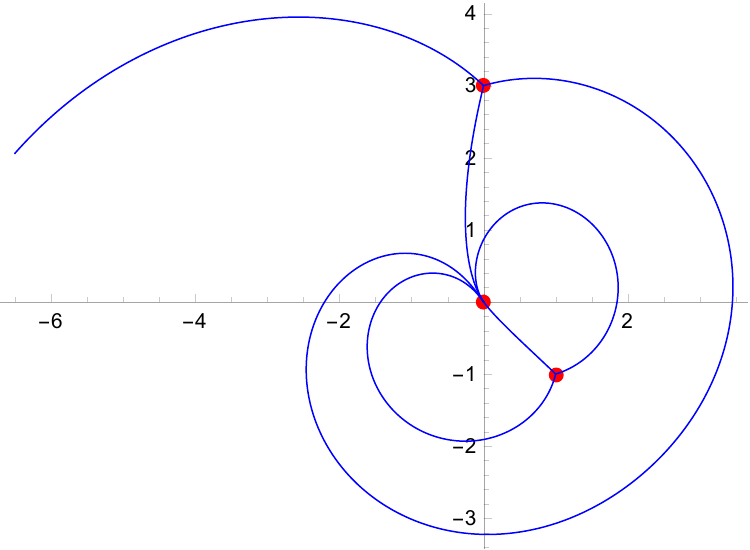}
\caption{$\theta_{0}=2\pi$ and $\theta=\frac{2\pi}{3}.$ Here
$\lambda^{2}=-2-i,a=3i,b=1-i.$ }%
\label{FIG3}%
\end{minipage}
\end{figure}

\begin{remark}
Proposition \ref{main} is in consistency with \cite[Proposition 1.2]{Atia2}.
Indeed, the structure of the critical graph of under Mobius transform (
$z=1/y$ ).
\end{remark}

\begin{proof}
[Proof of Proposition \ref{main}]Suppose that $\Re\left(  \frac{\lambda\left(
\sqrt{a}+\sqrt{b}\right)  ^{2}}{\sqrt{ab}}\right)  =0$ and there is no short
trajectory connecting $a$ and $b.$ The idea of the proof is to construct two
paths $\gamma_{1}$ and $\gamma_{2}$ joining $a$ and $b,$ not homotopic in $%
\mathbb{C}
\setminus\left\{  0\right\}  ,$ and such that%
\[
\Im\int_{\gamma_{1}}\dfrac{\lambda\sqrt{\left(  z-a\right)  \left(
z-b\right)  }}{z^{2}}dz\neq0,\Im\int_{\gamma_{2}}\dfrac{\lambda\sqrt{\left(
z-a\right)  \left(  z-b\right)  }}{z^{2}}dz\neq0,
\]
which contradicts Proposition \ref{real} and the fact that $\Re\left(
\frac{\lambda\left(  \sqrt{a}+\sqrt{b}\right)  ^{2}}{\sqrt{ab}}\right)  =0,$
see examples in \cite{paper1},\cite{Atia2}.

We first assume that $\Re\left(  \lambda\right)  =0.$ Since $\infty$ is a
double pole with negative residue, $\varpi\left(  \lambda,a,b,z\right)  $ has
a circular domain bounded by critical trajectories. More precisely, there
emanate from a zero of $\varpi\left(  \lambda,a,b,z\right)  ,$ for example
$a,$ two critical trajectories forming a loop $\Omega$ around $\infty$ (in
$\widehat{%
\mathbb{C}
}$). The third one, say $\gamma_{a},$ diverges to the origin. The interior
angle of the loop $\Omega$ at the vertices $a$ equals $\frac{4\pi}{3};$
applying Lemma \ref{teich}, we get
\[
-1=2+\sum n_{i}.
\]
It follows that $\Omega$ contains $0$ and $b$ inside, and then, all critical
trajectories emanating from $b$ stay in $\Omega.$ Again, with Lemma
\ref{teich}, there cannot exist a loop passing through $b,$ and then, we have
four trajectories that diverge to the origin. Assume that three of them
emanate from $b$ and diverge to the origin from the same direction. Then two
of them, spacing under angle $\theta\in\left\{  2\pi/3,4\pi/3\right\}  ,$
should form together with the origin, a domain $\Omega_{b}$ not intersecting
$\gamma_{a}.$ Applying Lemma \ref{teich}, we get the contradiction
\[
2=2+\sum n_{i}=\left(  2-\frac{3\theta}{2\pi}\right)  \in\left\{  0,1\right\}
.
\]
Then, there are exactly two trajectories $\gamma_{b,1}$ and $\gamma_{b,2}$
emanating from $b,$ and diverging to the origin from the same direction of
$\gamma_{a}.$ Consider a point $c\in\gamma_{a}$ close to the origin, in a way
that the two rays of the orthogonal trajectory $\gamma_{c}^{\perp}$ passing
through $c,$ tend to the origin in opposite directions. Then, $\gamma
_{c}^{\perp}$ intersects $\gamma_{b,1}$ and $\gamma_{b,2}$ respectively at two
distinct points, say $b_{1}$ and $b_{2}.$ Finally, set

\begin{itemize}
\item $\gamma_{1}$: the path formed by the part of $\gamma_{b,1}$ from $b$ to
$b_{1},$ the part of $\gamma_{c}^{\perp}$ from $b_{1}$ to $c,$ and, the part
of $\gamma_{a}$ from $c$ to $a.$

\item $\gamma_{2}$: the path formed by the part of $\gamma_{b,2}$ from $b$ to
$b_{2},$ the part of $\gamma_{c}^{\perp}$ from $b_{2}$ to $c,$ and, the part
of $\gamma_{a}$ from $c$ to $a.$
\end{itemize}

\begin{figure}[tbh]
\centering
\begin{tabular}
[c]{ll}%
\mbox{\begin{overpic}[scale=0.6]{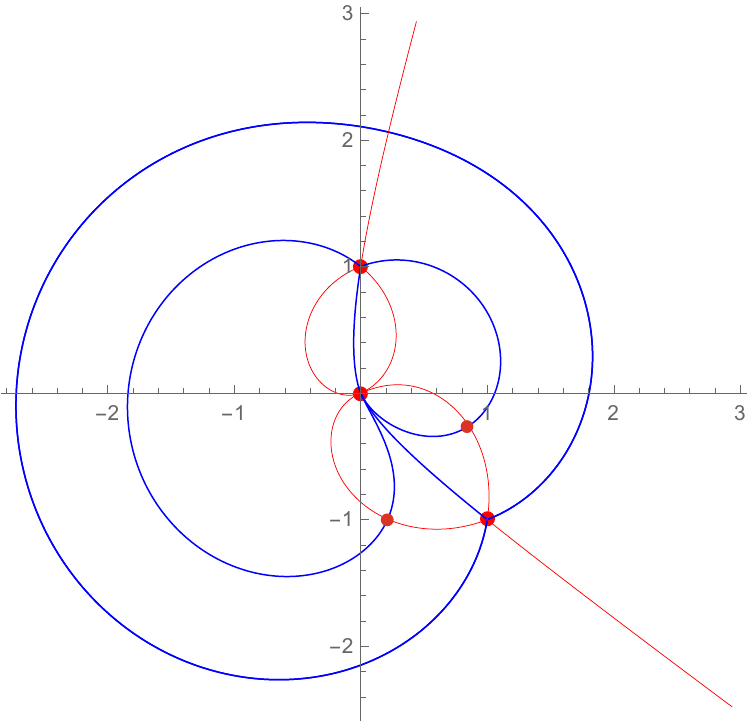}				
				\put(67,25){$a=c$ }
				\put(22,63){\small $\gamma_{b,1}$ }
				\put(61,61){\small $\gamma_{b,2}$ }
				\put(46,25){\small $b_{1}$ }
				\put(66,38){\small $b_{2}$ }
				\put(46,63){\small $b$ }
			
		\end{overpic}} &
\end{tabular}
\caption{Horizontal (blue) and vertical (red) trajectories of $\varpi\left(
i,i,1-i\right)  $ and the paths $\gamma_{1}$ and $\gamma_{2}.$ Here, we take,
without loss of the generality, $a=c.$ }%
\end{figure}See Figure \ref{FIG4}, where we take, without loss of the
generality $a=c.$ Clearly, $\gamma_{1}$ and $\gamma_{2}$ are two Jordan arcs
connecting $a$ and $b,$ not homotopic in $%
\mathbb{C}
\setminus\left\{  0\right\}  .$ Notice that, if $z\left(  s\right)  ,s\in
I\subset%
\mathbb{R}
$ is an arc of a vertical trajectory of $\varpi\left(  \lambda,a,b,z\right)
,$ then the function $s\longmapsto\mathcal{\Im}\int^{s}\lambda\frac
{\sqrt{\left(  z\left(  t\right)  -a\right)  \left(  z\left(  t\right)
-b\right)  }}{z\left(  t\right)  ^{2}}\,z^{\prime}\left(  t\right)  dt$ is
monotone. Integrating along each of the paths $\gamma_{1}$ and $\gamma_{2},$
we get :
\begin{align*}
\Im\int_{\gamma_{1}}\dfrac{\lambda\sqrt{\left(  z-a\right)  \left(
z-b\right)  }}{z^{2}}dz &  =\Im\int_{b_{1}}^{c}\dfrac{\lambda\sqrt{\left(
z-a\right)  \left(  z-b\right)  }}{z^{2}}dz\neq0;\\
\Im\int_{\gamma_{2}}\dfrac{\lambda\sqrt{\left(  z-a\right)  \left(
z-b\right)  }}{z^{2}}dz &  =\Im\int_{b_{2}}^{c}\dfrac{\lambda\sqrt{\left(
z-a\right)  \left(  z-b\right)  }}{z^{2}}dz\neq0,
\end{align*}
which violates the hypothesis.

The case $\Re\left(  \lambda\right)  \neq0$ is in the same vein since from
Lemma \ref{2TRAJ TO INF}, there are at list 4 critical trajectories diverging
to the origin.
\end{proof}

\begin{remark}
Proposition \ref{main} is in consistency with \cite[Proposition 1.2]{Atia2}.
Indeed, the structure of the critical graph of a quadratic differential is
invariant under Mobius transform; in particular, by the change of variable
$z=1/y$ in $\varpi\left(  \lambda,a,b,z\right)  ,$ we get the quadratic
differential studied in \cite{Atia2}.
\end{remark}

\begin{proposition}
\label{first prop}If equation (\ref{eq alg}) admits a real mother-body measure
$\mu,$ then:

\begin{itemize}
\item for some choice of the square root, $p-\sqrt{p^{2}-4r}\in%
\mathbb{R}
;$

\item any connected curve in the support of $\mu$ coincides with a short
trajectory of the quadratic differential
\begin{equation}
\varpi=-\frac{D\left(  z\right)  }{z^{4}}dz^{2}, \label{1qd}%
\end{equation}
where
\[
D\left(  z\right)  =\left(  p^{2}-4r\right)  z^{2}+2pqz+q^{2}%
\]
is the discriminant of the quadratic equation (\ref{eq alg}).
\end{itemize}
\end{proposition}

\begin{proof}
Solutions of (\ref{eq alg}) as a quadratic equation are%
\[
\mathcal{C}\left(  z\right)  =\frac{-\left(  pz+q\right)  +\sqrt{D\left(
z\right)  }}{2z^{2}},
\]
with some choice of the square root of the discriminant $\sqrt{D\left(
z\right)  }.$ If $\mu$ is a real mother-body measure of (\ref{eq alg}), then,
its Cauchy transform $\mathcal{C}_{\mu}$ is defined in $%
\mathbb{C}
\setminus$\emph{supp}$\left(  \mu\right)  $ by
\[
\mathcal{C}_{\mu}\left(  z\right)  =\int_{%
\mathbb{C}
\setminus\text{\emph{supp}}\left(  \mu\right)  }\frac{d\mu\left(  t\right)
}{z-t};
\]
it satisfies :%
\begin{equation}
\mathcal{C}_{\mu}\left(  z\right)  =\frac{\mu\left(
\mathbb{C}
\right)  }{z}+\mathcal{\allowbreak O}\left(  z^{-2}\right)  ,z\rightarrow
\infty; \label{cauchy cond 1}%
\end{equation}%
\begin{equation}
\mathcal{C}_{\mu}\left(  z\right)  \text{ is analytic in }%
\mathbb{C}
\setminus\emph{supp}\left(  \mu\right)  ; \label{cauchy cond 2}%
\end{equation}
and the inversion formula in the distributional sense :
\[
\mu=\frac{1}{\pi}\frac{\partial\mathcal{C}_{\mu}}{\partial\overline{z}}.
\]

\bigskip From the equality
\[
\mathcal{C}_{\mu}\left(  z\right)  =\frac{-p+\sqrt{p^{2}-4r}}{2z}%
+\mathcal{\allowbreak O}\left(  z^{-2}\right)  ,z\rightarrow\infty,
\]
we conclude that%
\[
\frac{-p+\sqrt{p^{2}-4r}}{2}=\mu\left(
\mathbb{C}
\right)  .
\]

Since the Cauchy transform $\mathcal{C}_{\mu}\left(  z\right)  $ of the
positive measure $\mu$ coincides a.e. in $%
\mathbb{C}
$ with an algebraic solution of the quadratic equation (\ref{eq alg}), it
follows that the support of this measure is a finite union of semi-analytic
curves and isolated points, see \cite[Theorem 1]{positive}. Let $\gamma$ be a
connected curve in the support of $\mu.$ For $t\in\gamma,$ we have
\[
\mathcal{C}_{\mu}^{+}\left(  t\right)  -\mathcal{C}_{\mu}^{-}\left(  t\right)
=\frac{\sqrt{D\left(  t\right)  }}{t^{2}}%
\]
From Plemelj-Sokhotsky's formula, we have
\[
\frac{1}{2\pi i}\left(  \mathcal{C}_{\mu}^{+}\left(  t\right)  -\mathcal{C}%
_{\mu}^{-}\left(  t\right)  \right)  dt\in%
\mathbb{R}
^{\ast},t\in\gamma,
\]
and then we get%
\[
-\frac{D\left(  t\right)  }{t^{4}}dt^{2}>0,t\in\gamma,
\]
which shows that $\gamma$ is a horizontal trajectory of the quadratic
differential (\ref{1qd}).

Since the Cauchy transform $\mathcal{C}_{\mu}\left(  z\right)  $ is well
defined and analytic in $%
\mathbb{C}
\setminus$\emph{supp}$\left(  \mu\right)  ,$ it follows that \emph{supp}%
$\left(  \mu\right)  $ should contain all branching points (zeros) of
(\ref{1qd}). Thus, any connected curve in the support of $\mu$ is a short
trajectory of the quadratic differential (\ref{1qd}).

For the proofs of the general cases, we refer the reader to \cite{amf rakh1},
\cite{pritsker},\cite{Shapiro},\cite{bullgard}.
\end{proof}

\section{Connection with Bessel polynomials $B_{n}^{\left(  \alpha\right)  }$}

Bessel polynomials $B_{n}^{\left(  \alpha\right)  }$ are given explicitly by
(see \cite{bessel})%

\begin{equation}
B_{n}^{\left(  \alpha\right)  }\left(  z\right)  =\sum_{k=0}^{n}\binom{n}%
{k}\left(  n+k+\alpha-2\right)  ^{\left(  k\right)  }z^{k},
\label{poly bessel}%
\end{equation}
where $\binom{n}{k}$ is a binomial coefficient, and $\left(  x\right)
^{\left(  k\right)  }$ as usual, means $x(x-1)...(x-k+1)$. Equivalently, these
polynomials can also be given by the Rodrigues formula
\[
B_{n}^{\left(  \alpha\right)  }\left(  z\right)  =z^{2-\alpha}e^{\frac{1}{z}%
}\frac{d}{dz^{n}}\left(  z^{2n+\alpha-2}e^{-\frac{1}{z}}\right)
\]
Clearly, polynomials $B_{n}^{\left(  \alpha\right)  }$ are entire functions of
the complex parameter $\alpha$. These polynomials fulfill the following
three-term recurrence relation and second degree differential equation%
\begin{align}
&  (n+\alpha-1)(2n+\alpha-2)B_{n}(z)\nonumber\\
&  =((2n+\alpha)(2n+\alpha-2)z+\alpha-2)(2n+\alpha-1)B_{n}(z)+n(2n+\alpha
)B_{n-1}(z), \label{recurr}%
\end{align}%
\begin{equation}
z^{2}B_{n}^{\prime\prime}(z)+(\alpha z+1)B_{n}^{\prime}(z)-n\left(
n+\alpha-1\right)  B_{n}(z)=0,\ n\geq1. \label{eq diff}%
\end{equation}

\bigskip A non-trivial asymptotic behavior can be obtained in the case of
varying coefficients. Namely, we will consider the sequences%
\[
B_{n}^{(\alpha_{n})},\quad\alpha_{n}/n\underset{n\longrightarrow\infty
}{\longrightarrow}A,
\]
where $A$ is fixed with assumption
\begin{equation}
A\in%
\mathbb{C}
\setminus\left\{  -1,-2\right\}  ,\Im A\geq0. \label{cond AB}%
\end{equation}

With each $B_{n}$ in (\ref{poly bessel}) we can associate its normalized
root-counting measure
\[
\mu_{n}=\mu\left(  B_{n}\right)  =\frac{\sum_{B_{n}\left(  a\right)  =0}%
\delta_{a}}{n},
\]
where $\delta_{a}$ is the Dirac measure supported at $a$ (the zeros are
counted with their multiplicity). The Cauchy transform of $\mu_{n}$ is given
by
\begin{equation}
\mathcal{C}_{\mu_{n}}\left(  z\right)  =\frac{B_{n}^{\prime}\left(  z\right)
}{nB_{n}\left(  z\right)  },B_{n}\left(  z\right)  \neq0.
\end{equation}

Combining (\ref{recurr}) and (\ref{eq diff}), and reasoning like in
\cite{weak} (and references therein), we get the following algebraic equation
\begin{equation}
z^{2}\mathcal{C}_{\mu}^{2}\left(  z\right)  +\left(  Az+1\right)
\mathcal{C}_{\mu}\left(  z\right)  -A-1=0. \label{alg eq}%
\end{equation}
The associated quadratic differential (\ref{1qd}) is:%
\begin{align*}
\varpi_{A}  &  =-\frac{D_{A}\left(  z\right)  }{z^{4}}dz^{2}=-\frac
{(A+2)^{2}z^{2}+2Az+1}{z^{4}}dz^{2}\\
&  =\left(  i(A+2)\right)  ^{2}\frac{\left(  z-\left(  \frac{1-i\sqrt{A+1}%
}{A+2}\right)  ^{2}\right)  \left(  z-\left(  \frac{1+i\sqrt{A+1}}%
{A+2}\right)  ^{2}\right)  }{z^{4}}dz^{2};
\end{align*}
its the critical graph will be denoted by $\Gamma_{A}.$ Straightforward
calculations show that Condition (\ref{cond nec}) is satisfied for $\varpi
_{A};$ more precisely, taking $\lambda=i(A+2),$ $a=\zeta_{-},$ and
$b=\zeta_{+},$ we have :
\[
\lambda\frac{\left(  \sqrt{a}\pm\sqrt{b}\right)  ^{2}}{\sqrt{ab}}=i\left(
A+2\right)  \frac{\left(  \sqrt{\zeta_{-}}\pm\sqrt{\zeta_{+}}\right)  ^{2}%
}{\sqrt{\zeta_{-}\zeta_{+}}}\in4i\left\{  1,-A-1\right\}  .
\]
From Proposition \ref{main}, quadratic differential $\varpi_{A}$ has two short
trajectories, when $A\in%
\mathbb{R}
\setminus\left\{  -2,-1\right\}  $ (see Figures \ref{3},\ref{4}), otherwise,
it has exactly one short trajectory $\delta_{A}$ (see Figures \ref{5}%
,\ref{6},\ref{7}).\begin{figure}[h]
\begin{minipage}[b]{0.3\linewidth}
\centering \includegraphics[scale=0.27]{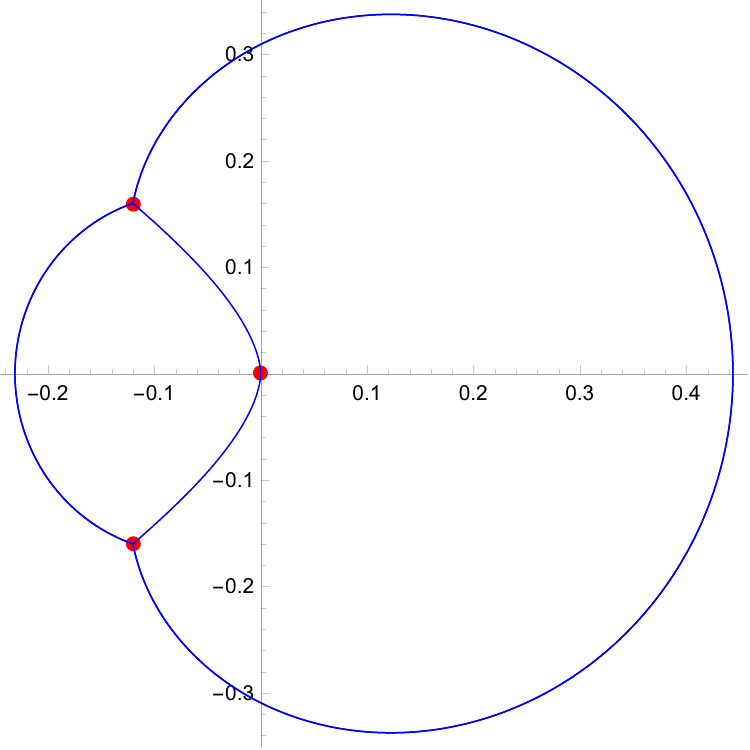}
\caption{$\Gamma_{3}.$}\label{3}%
\end{minipage}\hfill\begin{minipage}[b]{0.31\linewidth}
\centering \includegraphics[scale=0.25]{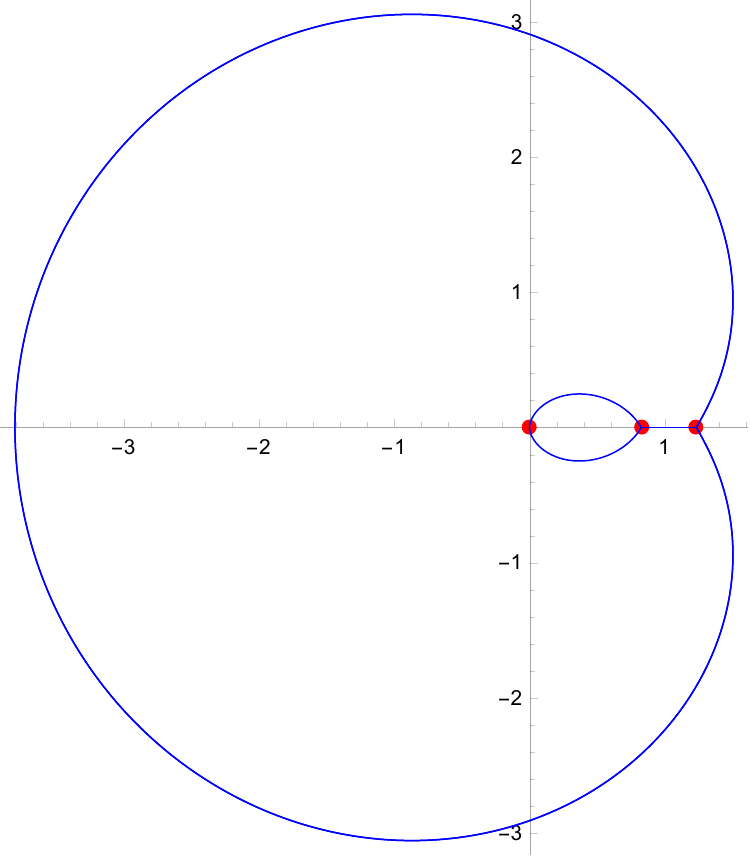}
\caption{$\Gamma_{-1.01.}$\label{4}}%
\end{minipage}
\begin{minipage}[b]{0.31\linewidth}
\centering \includegraphics[scale=0.25]{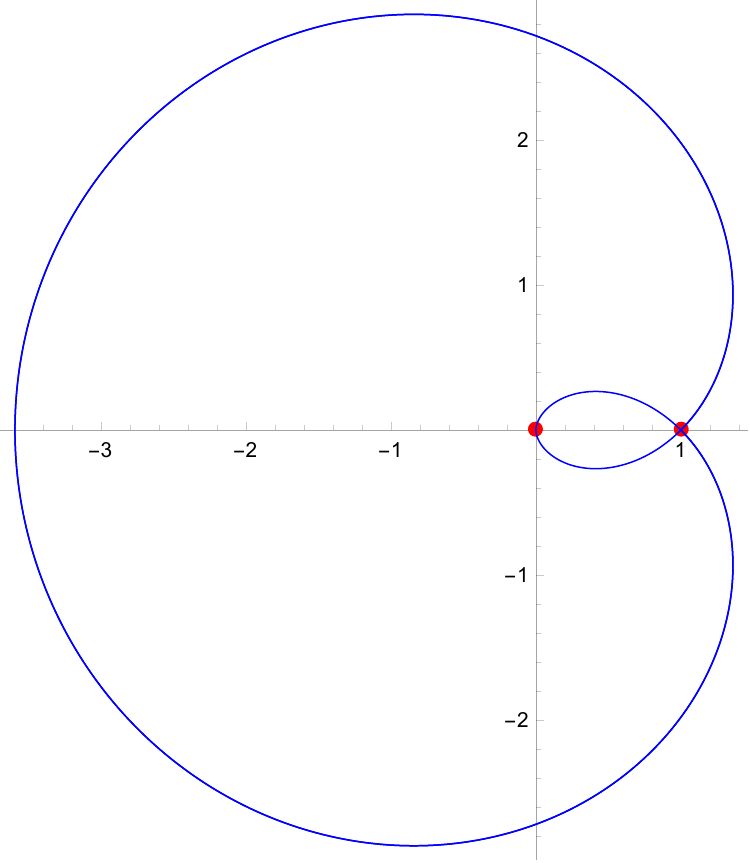}
\caption{$\Gamma_{-1}.$}\label{5}%
\end{minipage}\hfill\begin{minipage}[b]{0.5\linewidth}
\centering \includegraphics[scale=0.25]{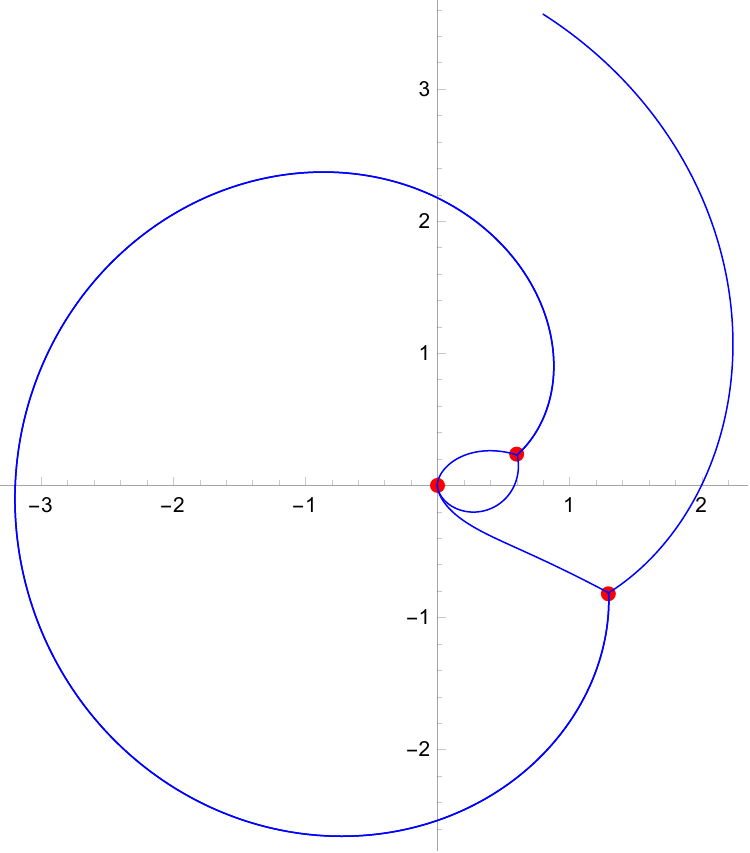}
\caption{$\Gamma_{-1+0.1i}.$}\label{6}%
\end{minipage}
\begin{minipage}[b]{0.3\linewidth}
\centering \includegraphics[scale=0.27]{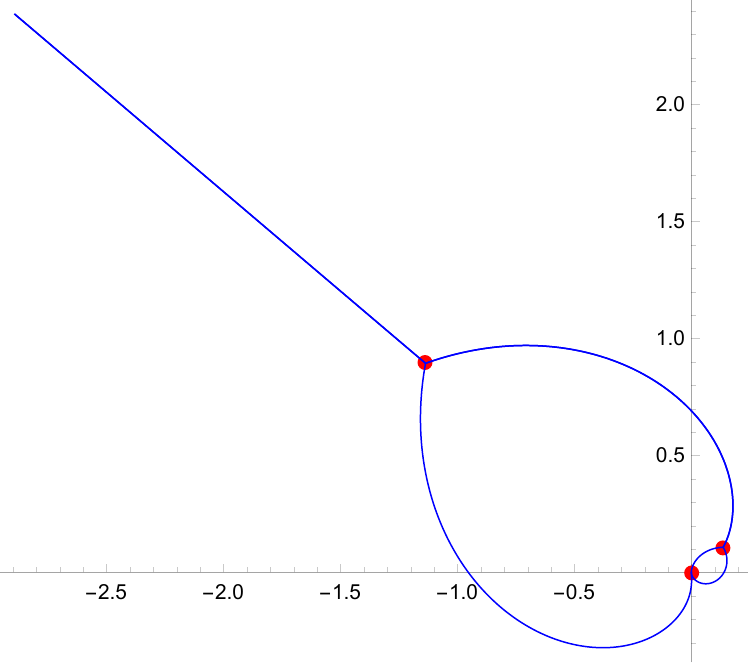}
\caption{$\Gamma_{-2+2i}.$}\label{7}%
\end{minipage}\hfill\end{figure}

The Cauchy transform $\mathcal{C}_{\mu}\left(  z\right)  $ of $\mu$ is given
by
\[
\mathcal{C}_{\mu}\left(  z\right)  =\frac{-\left(  Az+1\right)  +\sqrt
{D_{A}\left(  z\right)  }}{2z^{2}}.
\]
From Proposition \ref{first prop},\bigskip\ the branch of the square root
$\sqrt{D_{A}\left(  z\right)  }$ should be taken in $%
\mathbb{C}
\setminus\delta_{A},$ with normalization condition
\begin{equation}
\sqrt{D_{A}\left(  z\right)  }\sim(A+2)z,z\rightarrow\infty.
\label{Normalisation}%
\end{equation}
It follows from the behavior of $\mathcal{C}_{\mu}\left(  z\right)  $ at
$z=0:$
\[
\frac{-\left(  Az+1\right)  +\sqrt{D_{A}\left(  z\right)  }}{2z^{2}%
}=\allowbreak\left(  \frac{\sqrt{D_{A}\left(  0\right)  }}{2}-\frac{1}%
{2}\right)  \left(  \frac{1}{z^{2}}+\frac{A}{z}\right)  +\mathcal{O}\left(
1\right)  ,z\longrightarrow0\allowbreak,
\]
that
\begin{equation}
\sqrt{D_{A}\left(  0\right)  }=1. \label{sqrtof D in 0}%
\end{equation}
This condition gives a precious information about the homotopic class in $%
\mathbb{C}
\setminus\left\{  0\right\}  $ of the short trajectory $\delta_{A}.$ Indeed,
from the proof of Lemma \ref{real}, for any Jordan arc $\gamma$ connecting
$\zeta_{-}$ and $\zeta_{+}$ in $%
\mathbb{C}
\setminus\left\{  0\right\}  ,$ we have
\[
\int_{\gamma}\frac{\left(  \sqrt{D_{A}\left(  t\right)  }\right)  _{+}}{t^{2}%
}dt\in=\pm2i\pi\left\{
\begin{array}
[c]{c}%
1,\text{ \emph{if} }\sqrt{D_{A}\left(  0\right)  }=1\\
-A-1,\text{ \emph{if} }\sqrt{D_{A}\left(  0\right)  }=-1.
\end{array}
\right.
\]
More precisely, we have the following :

\begin{lemma}
Assume that zeros $\zeta_{-}$ and $\zeta_{+}$ of $D_{A}\left(  z\right)  $ are
not real, and let $\mathcal{J}_{A}$ be the set of all Jordan arcs connecting
them in $%
\mathbb{C}
\setminus\left\{  0\right\}  ,$ and not intersecting the semi-axis positive.
Then, for any Jordan arc $\gamma$ connecting $\zeta_{-}$ and $\zeta_{+}$ in $%
\mathbb{C}
\setminus\left\{  0\right\}  :$%
\[
\sqrt{D_{A}\left(  0\right)  }=\left\{
\begin{array}
[c]{c}%
1,\text{ \emph{if} }\gamma\in\mathcal{J}_{A}\\
-1,\text{ \emph{otherwise.}}%
\end{array}
\right.
\]
In particular, $\delta_{A}\in\mathcal{J}_{A}.$
\end{lemma}

\begin{proof}
Let $\gamma\in\mathcal{J}_{A}$ and consider the function
\[
g\left(  z\right)  =\frac{\sqrt{D_{A}\left(  z\right)  }}{\left(  A+2\right)
z+1}%
\]
defined and holomorphic in $%
\mathbb{C}
\setminus\left(  \gamma\cup\left\{  \frac{-1}{A+2}\right\}  \right)  ;$ in
particular, $g$ is continuous in $\left[  0,+\infty\right[  .$ Clearly,
$\sqrt{D_{A}\left(  0\right)  }\in\left\{  -1,1\right\}  .$ If $\sqrt
{D_{A}\left(  0\right)  }=-1,$ then with Condition (\ref{Normalisation}),
there exists $x>0,$ such that $\Re g\left(  x\right)  =0;$ which gives
$y\leq0,$ such that, $g^{2}\left(  x\right)  =y.$ Straightforward calculations
show that
\begin{equation}
\left(  A+2\right)  ^{2}x^{2}+2\left(  A-2\frac{y}{1-y}\right)  \allowbreak
x+1=0. \label{last}%
\end{equation}
Taking the real parts in (\ref{last}) and then the imaginary parts, we get the
contradiction
\[
\left(  \allowbreak\Im A\right)  ^{2}x+\frac{4}{1-y}=0.
\]
Thus, $\sqrt{D_{A}\left(  0\right)  }=1$ for $\gamma\in\mathcal{J}_{A}.$
\end{proof}

Generalized Laguerre polynomials $L_{n}^{\alpha_{n}}\left(  z\right)  $ are
given explicitly by (see \cite{szego}):
\begin{equation}
L_{n}^{(\alpha_{n})}(z)=\sum_{k=0}^{n}\binom{n+\alpha_{n}}{n-k}\frac{(-z)^{k}%
}{k!}, \label{laguerre1}%
\end{equation}
The asymptotic distribution of their zeros and asymptotic was studied in
\cite{Kuijlaars},\cite{weak},\cite{paper1}. The key feature in the study of
the weak asymptotic is the so-called Gonchar-Rakhmanov-Stahl (GRS) theory
\cite{gonchar},\cite{stahl}. The strong uniform asymptotic on the whole plane
are obtained by means of the Riemann--Hilbert steepest descent method of
Deift-Zhou \cite{deift}.

Generalized Bessel polynomials $B_{n}^{(\alpha)}$ are linked to re-scaled
generalized Laguerre polynomials by
\[
B_{n}^{(\alpha)}(z)=z^{n}L_{n}^{(-2n-\alpha+1)}\left(  \frac{2}{z}\right)
\,.
\]
The asymptotic distribution of their zeros and weak and strong asymptotic can
be derived by replacing $A\mapsto-(A+2)$ and $z\mapsto2/z.$ In particular, if
$A\notin%
\mathbb{R}
$ and $\gamma_{A}$ is the unique short trajectory of $\varpi_{A},$ then, the
sequence $\left(  \mu_{n}\right)  _{n}$ weakly converges to the measure $\mu$
absolutely continuous with respect to the arc-length measure :\bigskip\
\[
d\mu\left(  z\right)  =\frac{1}{2\pi i}\frac{\left(  \sqrt{D_{A}\left(
z\right)  }\right)  _{+}}{z^{2}}dz.
\]


\begin{acknowledgement}
\bigskip F.T has been partially supported by the research laboratory
\textquotedblleft Mathematics and Applications LR05ES14\textquotedblright%
\ from the Faculty of Sciences of Gabes, Tunisia. M.J.A would like to thank
the Deanship of Graduate Studies And Scientific Reasearch at Qassim University
for financial support (QU-APC-2025).
\end{acknowledgement}

\bigskip

\texttt{Mohamed Jalel Atia, Math Dept, College of Sciences, Qassim }

\texttt{Universty, Buraydah 51452. Saudi Arabia.}

\texttt{E-mail adress: M.Attia@qu.edu.sa}

\bigskip

\texttt{Faouzi Thabet, Higher Institute of Computer Science, Djerba Road Km }

\texttt{3, BP 283 Medenine 4100. Tunisia.}

\texttt{E-mail adress: faouzithabet@yahoo.fr}


\begin{thebibliography}{99}                                                                                               %


\bibitem {MR1929066}Jenkins,J.~A.: Univalent functions and conformal mapping,
Ergebnisse der Mathematik und ihrer Grenzgebiete. Neue Folge, Heft 18. Reihe:
Moderne Funktionentheorie, Springer-Verlag, Berlin, 1958.

\bibitem {Strebel:84}Strebel,K.: Quadratic differentials, Ergebnisse der
Mathematik und ihrer Grenzgebiete (3) [Results in Mathematics and Related
Areas (3)], vol.~5, Springer-Verlag, Berlin, 1984.

\bibitem {Kuijlaars}Kuijlaars,A. B. J. and McLaughlin,K. T.R.: Asymptotic zero
behavior of Laguerre polynomials with negative parameter, Constructive
Approximation 20 (4) (2004) 497-523.

\bibitem {amf rakh1}Mart\'{\i}nez-Finkelshtein,A. and . Rakhmanov,E. A.:
Critical measures, quadratic differentials, and weak limits of zeros of
Stieltjes polynomials, Commun. Math. Phys. vol. 302 (2011) 53-111.

\bibitem {weak}Mart\'{\i}nez-Finkelshtein,A., Mart\'{\i}nez-Gonzalez,P. and
Orive,R.: On asymptotic zero distribution of Laguerre and generalized Bessel
polynomials with varying parameters, J. Comp. Appl. Math. 133 (2001), 477--487.

\bibitem {deift}Deift, P.A.: Orthogonal polynomials and random matrices: a
Riemann-Hilbert approach. New York University Courant Institute of
Mathematical Sciences, New York (1999).

\bibitem {gonchar}Gonchar, A.A., Rakhmanov, E.A.: Equilibrium distributions
and degree of rational approximation of analytic functions. Math. USSR Sbornik
62(2), 305--348 (1987). (translation from Mat. Sb., Nov. Ser.134(176),
No.3(11), 306--352 (1987))

\bibitem {szego}Szego,G.: Orthogonal polynomials, fourth ed., Amer. Math. Soc.
Colloq. Publ., vol. 23, Amer. Math. Soc., Providence, RI, 1975.

\bibitem {bessel}Krall,H. L., and Frink,O.: A New Class of Orthogonal
Polynomials: The Bessel Polynomials, Trans.Am.Math.Society Vol. 65, No. 1
(Jan., 1949), pp. 100-115.

\bibitem {stahl}Stahl,H.: Orthogonal polynomials with complex-valued weight
function. I,II, Constr. Approx. 2 (1986), no. 3, 225--240, 241--251.

\bibitem {pritsker}Pritsker,I.E.: How to find a measure from its potential.
Com.Meth.Funct.Theory, Volume 8 (2008),No.2, 597-614.

\bibitem {positive}Bjork,J.-E., Borcea,J., and B\H{o}gvad,R.: Subharmonic
Configurations and Algebraic Cauchy Transforms of Probability Measures.
Notions of Positivity and the Geometry of Polynomials Trends in Mathematics
2011, pp 39-62.

\bibitem {paper1}Atia,M. J., Mart\'{\i}nez-Finkelshtein,A.,
Mart\'{\i}nez-Gonzalez,P., and Thabet,F.: Quadratic differentials and
asymptotics of Laguerre polynomials with varying complex parameters, J. Math.
Anal. Appl. 416 (2014), 52--80.

\bibitem {Atia2}Atia,M. J., and Thabet,F.: Quadratic differentials
A(z-a)(z-b)dz\symbol{94}2/(z-c)\symbol{94}2 and algebraic Cauchy transform,
Cz.Math.Journal., 66 (141) (2016), 351--363.

\bibitem {Shapiro}B\H{o}gvad,R., and Shapiro, B.: On motherbody measures and
algebraic Cauchy transform, Enseign. Math. 62 (2016), 117-142

\bibitem {bullgard}Bergkvist,T., and H. Rullg\aa rd,H.: On polynomial
eigenfunctions for a class of differential operators. Math. Res. Lett. 9
(2002), 153-171.
\end{thebibliography}
\end{document}